\documentclass[preprint,12pt]{elsarticle}
\newtheorem{theorem}{Theorem}[section]
\newtheorem{Cor}{Corollary}[theorem]
\newtheorem{Prop}{Proposition}[section]
\newtheorem{Def}{Definition}[section]
\newtheorem{Lem}{Lemma}[section]
\newtheorem{Exa}{Example}[section]
\newtheorem{Rem}{Remark}[section]
\newtheorem{proof}{Proof}

\usepackage{amssymb}
\usepackage[all]{xy}

\usepackage[nodots,nocompress]{numcompress}

\biboptions{super}

\begin{document}

\begin{frontmatter}



\title{Combinatorial aspects of dynamical Yang-Baxter maps and dynamical braces}


\author{Diogo Kendy Matsumoto$^{a}$ }

\address{$^{a}$Department of Mathematics, Fundamental Science and Engineering, Waseda University, 3-4-1 Okubo, Shinjuku-ku, Tokyo 169-8555, Japan}
\ead{diogo-swm@akane.waseda.jp}

\begin{abstract}
In this article we propose an algebraic system, which is an abelian group $(A,+)$ with a family of non-associative and non-(left)distributive multiplications $\{\cdot_{\lambda}\}_{\lambda\in H}$. We call this algebraic system dynamical brace. 
The dynamical brace corresponds to a certain dynamical Yang-Baxter map (which is left nondegenerate and satisfy the unitary condition). 
Combinatorial aspects of the dynamical brace give us a correspondence between the dynamical brace and a certain family of subsets of $A\rtimes Aut(A)$. 
From this viewpoint we give an interpretation and examples of the dynamical brace. 


\end{abstract}


\end{frontmatter}

% \linenumbers

\section{Introduction}
 A dynamical Yang-Baxter map (DYB map) is a set-theoretical solution of a dynamical Yang-Baxter equation. The DYB map is given as a generalization of the Yang-Baxter map (YB map) by Y.Shibukawa[1]. 
The YB map is a set-theoretical solution of the Yang-Baxter equation, that plays an important role in many areas, and has closely relations with bijective 1-cocycles[2], semigroups of I-type[3], and many other things[4,5,6].  

Let $X,H$ be non-empty sets and $\phi$ a map from $H\times X$ to $H$. Then the DYB map associated with $X,H,\phi$ is a map $R(\lambda):X\times X\rightarrow X\times X$, that satisfies the following relation for all $\lambda \in H$.
\begin{equation}
   R_{23}(\lambda)R_{13}(\phi (\lambda ,X^{(2)}))R_{12}(\lambda)
                           =R_{12}(\phi (\lambda ,X^{(3)}))R_{13}(\lambda)R_{23}(\phi (\lambda ,X^{(1)}))                            
\end{equation}(See Definition 2.1).
In [7] Y.Shibukawa describes DYB maps with the invariance condition when $X=H$ is a left quasigroup using ternary system.

 Our purpose is to obtain solutions of the DYB map (in specially when $X\not= H$ ), and is to show structures of the DYB map (moreover give a meaning of DYB maps). 
 
For this purpose, in this article firstly we propose an algebraic system. We call this algebraic system dynamical brace(d-brace). D-braces correspond to a certain DYB maps, which is right nondegenerate and satisfies the unitary condition. The definition of d-brace is as follows.
 
Let $H$ be a non-empty set, $(A,+)$ an abelian group with the family of multiplications $\{ \cdot_{\lambda}:A\times A\rightarrow A \}_{\lambda\in H}$ and $\phi$ a map from $H\times A$ to $H$. Then $(A,H,\phi;+,\{ \cdot_{\lambda} \}_{\lambda \in H})$ is a d-brace if the following conditions are satisfied for all $(\lambda,a,b,c)\in H\times A\times A\times A$. 
\begin{enumerate}
  \item $(a+b)\cdot_{\lambda} c=a\cdot_{\lambda} c +b\cdot_{\lambda} c$ \quad (Right distributive raw),
  \item $a\cdot_{\lambda} (b\cdot_{\lambda} c+b+c)= (a\cdot_{\phi(\lambda,c)} b)\cdot_{\lambda} c+a\cdot_{\phi(\lambda,c)} b+a\cdot_{\lambda} c$,
  \item The map $\gamma_{\lambda}(b) :a \mapsto a\cdot_{\lambda} b+a $ is bijective.
\end{enumerate} 
(See Definition 3.2). The d-brace is a generalization of the brace, that proposed by W.Rump to construct YB map in [6].

Secondly, we state a combinatorial aspect of d-braces and DYB maps . 
We express the d-brace $(A,H,\phi;+,\{\cdot_{\lambda}\}_{\lambda\in H} )$ as a certain family of subsets of $A\rtimes Aut(A)$, and we construct a directed graph associated with this family of subsets. 
From this we obtain a good viewpoint of d-braces, and many examples. 
(This graph is a kind of incident geometry).

The organization of the article is as follows.
 Section 2 we give the definition of DYB maps and basic notions.
In section 3, we state the relation between YB maps and braces that proved by W.Rump[6]. Next we introduce the d-brace, its properties, and prove that d-braces correspond to some DYB maps. Moreover we give the relation between d-braces and braces in special case.
Section 4 and 5, we prove that d-brace structures over an abelian group $(A,+)$ corresponds to a certain family of subsets of $A\rtimes Aut(A)$, and construct directed graphs of d-braces.
From this graphs we obtain many informations of d-brace. Lastly we give several examples of d-braces.

\section{Dynamical Yang-Baxter maps}
In this section we introduce basic notions and results of DYB maps.
Let $X,H$ be non-empty sets and $\phi$ a map from $H\times X$ to $H$. We call elements of $H$ dynamical parameters.
\begin{Def}
{\rm A map $R(\lambda):X\times X \rightarrow X\times X(\lambda \in H)$is a dynamical Yang-Baxter map(DYB map) associated with $X,H,\phi$ if
$R(\lambda)$ satisfies the following equation on $X\times X\times X$ for all $\lambda \in H$.
\begin{equation}
   R_{23}(\lambda)R_{13}(\phi (\lambda ,X^{(2)}))R_{12}(\lambda)
                           =R_{12}(\phi (\lambda ,X^{(3)}))R_{13}(\lambda)R_{23}(\phi (\lambda ,X^{(1)}))
\end{equation}

here $R_{12}(\lambda),R_{12}(\phi (\lambda ,X^{(3)})) ,\cdots$ are maps from $X\times X\times X$ to $X\times X\times X$ defined by 
\[
  R_{12}(\lambda)(a,b,c)=(R(\lambda)(a,b),c) ,
\]
\[
  R_{12}(\phi (\lambda ,X^{(3)}))(a,b,c)=(R(\phi (\lambda ,c))(a,b),c)  \quad (a,b,c\in X).
\]}
\end{Def}
As a special case of DYB maps, we can define Yang-Baxter maps.
\begin{Def}
{\rm A map $R:X\times X \rightarrow X\times X$ is a Yang-Baxter map(YB map) if $R$ satisfies the following equation on $X\times X\times X$
\begin{equation}
 R_{23}R_{13}R_{12} = R_{12}R_{13}R_{23}
\end{equation}
here $R_{ij}$ are defined similarly in the definition above. }
\end{Def}
As can be seen from the definitions above, a YB map is just a DYB map which is independent of the dynamical parameter.

We represent a map $R(\lambda): X\times X \rightarrow X\times X$ $(\lambda \in H)$ by
\begin{equation}
 R(\lambda)(a,b)=(\mathfrak{R}_{b}^{\lambda}(a), \mathfrak{L}_{a}^{\lambda}(b)), \quad (\lambda,a,b) \in H\times X\times X
\end{equation}
For $(a,\lambda) \in X\times H$, we define maps $\mathfrak{L}^{\lambda}_{a}:X\rightarrow X, \mathfrak{R}^{\lambda}_{a}:X\rightarrow X$ by
\begin{equation}
 \mathfrak{L}^{\lambda}_{a}:b\mapsto \mathfrak{L}^{\lambda}_{a}(b) , \mathfrak{R}^{\lambda}_{a}:b\mapsto \mathfrak{R}^{\lambda}_{a}(b).
\end{equation}
For $\lambda \in H$, we set $\mathfrak{L}^{\lambda}:X\times X \rightarrow X, \mathfrak{R}^{\lambda}:X\times X \rightarrow X$ by
\begin{equation}
 \mathfrak{L}^{\lambda}:(a,b)\mapsto \mathfrak{L}^{\lambda}_{a}(b) , \mathfrak{R}^{\lambda}:(a,b)\mapsto \mathfrak{R}^{\lambda}_{b}(a).
\end{equation}
Let $\mathfrak{L}$ be a map $\lambda \mapsto \mathfrak{L}^{\lambda}$ and $\mathfrak{R}$ a map $\lambda \mapsto \mathfrak{R}^{\lambda}$.

We obtain the next lemma.

\begin{Lem}
{\rm A map $R(\lambda):X\times X\rightarrow X\times X(\lambda\in H)$ associated with X,H,$\phi$ is DYB map if and only $\mathfrak{L},\mathfrak{R}$ satisfy 
 (7),(8),(9) for all $(\lambda,a,b,c)\in H\times X\times X\times X$.  }
    \begin{eqnarray}
       \mathfrak{L}_{a}^{\lambda} \cdot \mathfrak{L}_{b}^{\phi(\lambda ,a)}
      =\mathfrak{L}_{\mathfrak{L}_{a}^{\lambda}(b)}^{\lambda} \cdot \mathfrak{L}_{\mathfrak{R}_{b}^{\lambda}(a)}^{\phi (\lambda,\mathfrak{L}_{a}^{\lambda}(b))} ,
    \end{eqnarray}
    \begin{eqnarray}
      \mathfrak{R}_{(\mathfrak{L}_{\mathfrak{R}_{b}^{\lambda}(a)}^{\phi(\lambda,\mathfrak{L}_{a}^{\lambda}(b))}(c))}^{\lambda} \cdot \mathfrak{L}_{a}^{\lambda}(b)
      =\mathfrak{L}_{(\mathfrak{R}_{\mathfrak{L}_{b}^{\phi(\lambda,a)}(c)}^{\lambda}(a))}^{\phi(\lambda,\mathfrak{L}_{a}^{\lambda}\mathfrak{L}_{b}^{\phi(\lambda,a)}(c))} \cdot \mathfrak{R}_{c}^{\phi(\lambda,a)}(b) ,
    \end{eqnarray}
    \begin{eqnarray}
     \mathfrak{R}_{c}^{\phi(\lambda,\mathfrak{L}_{a}^{\lambda}(b))} \cdot \mathfrak{R}_{b}^{\lambda}(a)
      =\mathfrak{R}_{(\mathfrak{R}_{c}^{\phi(\lambda,a)}(b))}^{\phi(\lambda,\mathfrak{L}_{a}^{\lambda}\mathfrak{L}_{b}^{\phi(\lambda,a)}(c))} \cdot \mathfrak{R}_{\mathfrak{L}_{b}^{\phi(\lambda,a)}(c)}^{\lambda}(a) .
    \end{eqnarray}
\end{Lem}
\begin{proof}
{\rm The proof is straightforward. \qed }
\end{proof}
We call $R(\lambda)$ that satisfies conditions (7),(8),(9) a DYB map associated with $X,H,\phi$.
\begin{Def}
{\rm Let $R(\lambda)$ be a DYB map associated with $X,H,\phi$.
  \begin{enumerate}
   \item We say that $R(\lambda)$ is left nondegenerate if the map $\mathfrak{R}_{a}^{\lambda}$ is bijection,
         and $R(\lambda)$ is called right nondegenerate if the map $\mathfrak{L}_{b}^{\lambda}$ is bijection for all $(\lambda,a,b) \in H\times X\times X$.
         When $R(\lambda)$ is left and right nondegenerate we call it simply nondegenerate.
   \item We say that $R(\lambda)$ satisfies unitary condition if $R(\lambda)$ satisfies $PR(\lambda)PR(\lambda)$ \\ $=id_{X\times X}$($\forall \lambda \in H$). When DYB map satisfies unitary condition we call it simply unitary DYB map.
   \item Next condition about a map $\phi:H \times X \rightarrow H$ is called the weight-zero condition,
   \[
      \phi(\phi(\lambda,a),b) = \phi(\phi(\lambda,\mathfrak{L}^{\lambda}_{a}(b) ),\mathfrak{R}^{\lambda}_{b}(a) ),  \quad \mbox{for all }(\lambda,a,b) \in H\times X\times X.
   \]
  \end{enumerate} }
\end{Def}

\begin{Lem}    
{\rm A DYB map $R(\lambda):X\times X\rightarrow X\times X(\lambda\in H)$ associated with $X,H,\phi$ satisfies unitary condition if and only if $\mathfrak{L},\mathfrak{R}$ satisfy
 (10),(11) for all $(\lambda,a,b)\in H\times X\times X$. }
    \begin{equation}
     \mathfrak{L}_{\mathfrak{L}_{a}^{\lambda}(b)}^{\lambda} \cdot \mathfrak{R}_{b}^{\lambda}(a) = a,
    \end{equation}
    \begin{equation}
     \mathfrak{R}_{\mathfrak{R}_{b}^{\lambda}(a)}^{\lambda} \cdot \mathfrak{L}_{a}^{\lambda}(b) = b.
    \end{equation} 
\end{Lem}
\begin{proof}
{\rm The proof is straightforward. \qed }
\end{proof}

\begin{Exa}
{\rm \begin{enumerate} 
 \item Let $X$ be a non-empty set, $id_{X\times X}$ the identity map. Then $(X,id_{X\times X})$ is a unitary YB map. We call this YB map the trivial solution.
  \item (Lyubashenko[2])Let $X$ be a non-empty set, $r:X\times X\rightarrow X\times X, (a,b)\mapsto (\mathfrak{R}(a),\mathfrak{L}(b))$ (here $\mathfrak{L},\mathfrak{R}$ is maps from $X$ to $X$). Suppose that $\mathfrak{L}$ and $\mathfrak{R}$ are bijection. Then $(X,r)$ is a YB map if and only if $\mathfrak{LR}=\mathfrak{RL}$. Moreover $(X,r)$ satisfies the unitary condition if and only if $\mathfrak{R}=\mathfrak{L}^{-1}$.
We call this solution $(X,r)$ a permutation solution.
 \end{enumerate} }
\end{Exa}

 The following proposition gives relations between two DYB maps associated with distinct spaces, and a way to construct new solution.
\begin{Prop}
{\rm \begin{enumerate}
  \item Let $H$ be a non-empty set and $R'(\lambda)$ a DYB map associated with $X,H',\phi$.
        If there exist maps $\psi:H\rightarrow H'$C$\rho :H'\rightarrow H(\psi \rho =id_{H'})$,
        then the map $R(\lambda):X\times X\rightarrow X\times X(\lambda \in H)$, $R(\lambda)=R'(\psi(\lambda))$
        is DYB map associated with $X,H,\rho \phi(\psi \times id_{X})$
  \item Let $X$ be a non-empty set and $R'(\lambda)$ a DYB map associated with $X',H,\phi$.
        If there exist maps $\rho:X'\rightarrow X$C$\psi:X\rightarrow X' (\psi \rho=id_{X'})$,
        then the map $R(\lambda):X\times X\rightarrow X\times X(\lambda \in H), R(\lambda)=(\rho\times \rho)R'(\lambda)(\psi\times \psi)$
        is DYB map associated with $X,H,\phi(id_{X} \times \psi)$.
  \item Let $R_{X}(\lambda)$ be a DYB map associated with $X,H,\phi$, and $R_{Y}(\mu)$ a DYB map associated with $Y,I,\psi$. Then $(X\times Y,H\times I,\phi\times\psi,R_{X}(h)\times R_{Y}(i))$ is also 
        DYB map associated with $X\times Y,H\times I, \phi\times\psi$.  
 \end{enumerate} }
\end{Prop}

\begin{Def}
{\rm Let $R(\lambda)$ be a DYB map associated with $X,H,\phi$ and $R^{'}(\lambda^{'})$ a DYB maps associated $X^{'},H^{'},\phi^{'}$. $R(\lambda)$ is equivalent to $R^{'}(\lambda^{'})$ if and only if there exist two maps $F:X\rightarrow X^{'}$, $p:H\rightarrow H^{'}$ such that 
 \begin{enumerate}
  \item $p\phi = \phi^{'}(p\times F)$
  \item $(F\times F)R(\lambda) = R^{'}(p(\lambda) )(F\times F)$
 \end{enumerate} }
\end{Def}

\begin{theorem}
{\rm Let $\mathfrak{L}^{\lambda}_{a}:X \rightarrow X$ be a bijective map for all $(a,\lambda) \in X\times H$,
and $\displaystyle \mathfrak{R}^{\lambda}_{b}(a):=(\mathfrak{L}^{\lambda}_{\mathfrak{L}^{\lambda}_{a}(b)})^{-1}(a)$. 
Suppose that the maps $\mathfrak{L}^{\lambda}_{a} , \mathfrak{R}^{\lambda}_{b}$ satisfy the condition (7) of Lemma 2.1. 
Then a map $R(\lambda):X\times X\rightarrow X\times X$ defined by 
\[
 R(\lambda)(a,b):=( \mathfrak{R}^{\lambda}_{b}(a), \mathfrak{L}^{\lambda}_{a}(b) )=( (\mathfrak{L}^{\lambda}_{\mathfrak{L}^{\lambda}_{a}(b)})^{-1}(a) , \mathfrak{L}^{\lambda}_{a}(b) )
\] is a right nondegenerate unitary DYB map.}
\end{theorem}
\begin{proof}
{\rm First we shows that the condition (8) follows from (7).\\
Put $\displaystyle A=\mathfrak{L}^{\phi(\lambda,\mathfrak{L}^{\lambda}_{a}(b))}_{\mathfrak{R}^{\lambda}_{b}(a)}(c),B=\mathfrak{L}^{\lambda}_{a}\mathfrak{L}^{\phi(\lambda,a)}_{b}(c).$ 
Then
\begin{eqnarray*}
 \mbox{LHS\ of\ (8)}
 &=& \mathfrak{R}^{\lambda}_{A}\mathfrak{L}^{\lambda}_{a}(b)\\
 &=& (\mathfrak{L}^{\lambda}_{\mathfrak{L}^{\lambda}_{\mathfrak{L}^{\lambda}_{a}(b)}(A)})^{-1} \mathfrak{L}^{\lambda}_{a}(b)\\
 &=&(\mathfrak{L}^{\lambda}_{B})^{-1} \mathfrak{L}^{\lambda}_{a}(b),
\end{eqnarray*}
and 
\begin{eqnarray*}
 \mbox{RHS\ of\ (8)}
 &=& \mathfrak{L}^{\phi(\lambda,B)}_{(\mathfrak{R}^{\lambda}_{\mathfrak{L}^{\phi(\lambda,a)}_{b}(c)}(a))} \mathfrak{R}^{\phi(\lambda,a)}_{c}(b)\\
 &=& \mathfrak{L}^{\phi(\lambda,B)}_{(\mathfrak{L}^{\lambda}_{B})^{-1}(a)} (\mathfrak{L}^{\phi(\lambda,a)}_{\mathfrak{L}^{\phi(\lambda,a)}_{b}(c)})^{-1} (b).
\end{eqnarray*}
So we must prove $\displaystyle(\mathfrak{L}^{\lambda}_{B})^{-1} \mathfrak{L}^{\lambda}_{a}(b) = \mathfrak{L}^{\phi(\lambda,B)}_{(\mathfrak{L}^{\lambda}_{B})^{-1}(a)} (\mathfrak{L}^{\phi(\lambda,a)}_{\mathfrak{L}^{\phi(\lambda,a)}_{b}(c)})^{-1} (b).$
We have
\begin{eqnarray*}
 (\mathfrak{L}^{\lambda}_{a})^{-1} \mathfrak{L}^{\lambda}_{B} \mathfrak{L}^{\phi(\lambda,B)}_{(\mathfrak{L}^{\lambda}_{B})^{-1}(a)} (\mathfrak{L}^{\phi(\lambda,a)}_{\mathfrak{L}^{\phi(\lambda,a)}_{b}(c)})^{-1} (b)
 &=& (\mathfrak{L}^{\lambda}_{a})^{-1} (\mathfrak{L}^{\lambda}_{a} \mathfrak{L}^{\phi(\lambda,a)}_{\mathfrak{R}^{\lambda}_{(\mathfrak{L}^{\lambda}_{B})^{-1}(a) }(B) }) (\mathfrak{L}^{\phi(\lambda,a)}_{\mathfrak{L}^{\phi(\lambda,a)}_{b}(c)})^{-1} (b)\\
 &=& \mathfrak{L}^{\phi(\lambda,a)}_{\mathfrak{R}^{\lambda}_{(\mathfrak{L}^{\lambda}_{B})^{-1}(a)}(B)} (\mathfrak{L}^{\phi(\lambda,a)}_{\mathfrak{L}^{\phi(\lambda,a)}_{b}(c)})^{-1}(b)\\
 &=& b.
\end{eqnarray*}
Next we show that the condition (9) follows from (7) \\
Put $\displaystyle X=\mathfrak{L}^{\lambda}_{a}(b) , Y=\mathfrak{L}^{\phi(\lambda,\mathfrak{L}^{\lambda}_{a}(b) )}_{\mathfrak{R}^{\lambda}_{b}(a)}(c),
Z=\mathfrak{L}^{\lambda}_{X}(Y).$ Then 
\begin{eqnarray*}
 \mbox{LHS \ of\ (9)}
 &=& (\mathfrak{L}^{\phi(\lambda,\mathfrak{L}^{\lambda}_{a}(b))}_{\mathfrak{L}^{\phi(\lambda,\mathfrak{L}^{\lambda}_{a}(b))}_{\mathfrak{R}^{\lambda}_{b}(a)}(c) }  )^{-1} (\mathfrak{L}^{\lambda}_{\mathfrak{L}^{\lambda}_{a}(b)} )^{-1}(a)\\
 &=& (\mathfrak{L}^{\lambda}_{X} \mathfrak{L}^{\phi(\lambda,X)}_{Y})^{-1}(a)\\
 &=& (\mathfrak{L}^{\lambda}_{Z } \mathfrak{L}^{\phi(\lambda,Z )}_{\mathfrak{R}^{\lambda}_{Y}(X)}  )^{-1} (a)\\
 &=& (\mathfrak{L}^{\phi(\lambda,Z)}_{\mathfrak{R}^{\lambda}_{Y}(X)}  )^{-1} (\mathfrak{L}^{\lambda}_{Z})^{-1}(a)\\
\end{eqnarray*}
 ($\mathfrak{R}^{\lambda}_{Y}(X)$ =LHS of(8) ) 
\begin{eqnarray*}
 &=& (\mathfrak{L}^{\phi(\lambda,Z)}_{\mathfrak{L}^{\phi(\lambda,Z)}_{\mathfrak{R}^{\lambda}_{\mathfrak{L}^{\phi(\lambda,a)}_{b}(c)}(a) }\mathfrak{R}^{\phi(\lambda,a)}_{c}(b) }  )^{-1} (\mathfrak{L}^{\lambda}_{Z})^{-1}(a)\\
 &=& \mathfrak{R}^{\phi(\lambda,Z)}_{\mathfrak{R}^{\phi(\lambda,a)}_{c}(b)} \mathfrak{R}^{\lambda}_{\mathfrak{L}^{\phi(\lambda,a)}_{b}(c)}(a)\\
 &=& \mbox{RHS\ of\ (9)}.
\end{eqnarray*} \qed }
\end{proof}

Next we consider a commutative binary operator define on a non-empty set $X$.
That is, $+:X\times X\rightarrow X$ and $x+y=y+x$ for all $x,y\in X$.
(The associativity of $+$ is not assumed here).
\begin{Cor}
{\rm Let $X=(X,+)$ be a non-empty set with a commutative binary operation $+$, and bijective maps $\mathfrak{L}^{\lambda}_{a}:X\rightarrow X$ satisfies
\begin{eqnarray}
   \mathfrak{L}_{a}^{\lambda} \cdot \mathfrak{L}_{b}^{\phi(\lambda ,a)}=\mathfrak{L}_{\mathfrak{L}_{a}^{\lambda}(b)+a}^{\lambda}
\end{eqnarray}$($for all $\lambda \in H , a,b\in X)$. 
Then the maps $R(\lambda):X\times X\rightarrow X\times X$ 
\[
 R(\lambda)(a,b)=( \mathfrak{R}^{\lambda}_{b}(a), \mathfrak{L}^{\lambda}_{a}(b) ):=( (\mathfrak{L}^{\lambda}_{\mathfrak{L}^{\lambda}_{a}(b)})^{-1}(a) , \mathfrak{L}^{\lambda}_{a}(b) )
\] give a right nondegenerate unitary DYB map associated $X,H,\phi$.}
\end{Cor}
\begin{proof}
{\rm We shows that the condition(7) follows from condition(12).
\[
\mbox{RHS of (7)}= \mathfrak{L}^{\lambda}_{\mathfrak{L}^{\lambda}_{\mathfrak{L}^{\lambda}_{a}(b)} (\mathfrak{R}^{\lambda}_{b}(a)) +\mathfrak{L}^{\lambda}_{a}(b) } =\mathfrak{L}^{\lambda}_{a+\mathfrak{L}^{\lambda}_{a}(b) }= \mbox{LHS of (7)}.
\] \qed }
\end{proof}

\section{Braces and dynamical braces.}
In this section, we begin with to introduce a relation between braces and YB maps that proved in [6] by W.Rump.

\begin{Def}
{\rm Let $A=(A,+)$ be an abelian group with a multiplication $\cdot :A\times A\rightarrow A$. We call $(A,+,\cdot)$ a brace if the following conditions are satisfied for all $a,b,c\in A$. 
\begin{enumerate}
  \item $(a+b)\cdot c=a\cdot c +b\cdot c$ \quad (Right distributive law),
  \item $a\cdot (b\cdot c+b+c)= (a\cdot b)\cdot c+a\cdot b+a\cdot c$,
  \item The map $\gamma(b):a \mapsto a\cdot b+a $ is bijective.
\end{enumerate} }
\end{Def}

\begin{Prop}
{\rm An abelian group $A=(A,+)$ with a right distributive multiplication is a brace if and only if $A$ is a group with respect to the operation $a*b:=a\cdot b+a+b$ for all $a,b \in A$. } 
\end{Prop}

\begin{Prop}
{\rm
Let $(A,+,\cdot)$ be a brace and $0$ the unit of the abelian group $(A,+)$.Then 
\[
0\cdot a= a\cdot 0=0 \ (\mbox{for all }a\in A).
\]}
\end{Prop}
\begin{proof} 
{\rm 1. $0\cdot a=0$ is trivial. \\
2. $a\cdot0=a\cdot(0\cdot0+0+0)=(a\cdot0)\cdot0+a\cdot0+a\cdot0$,  
hence $\gamma(0)(a\cdot0)= (a\cdot0)\cdot0+a\cdot0= 0=0\cdot0+0= \gamma(0)(0)$.
Therefore we obtain $a\cdot0=0$ by using bijectivity of $\gamma(0)$. \qed}
\end{proof}

\begin{Exa}
{\rm 1. Abelian group $(A,+)$ with a multiplication $a\cdot b=0$ is a brace ($a,b\in A$). We call this $(A,+,\cdot)$ trivial brace. \\
2. Let $R=(R,+,\cdot)$ be a ring and $Jac(R)$ a Jacobson radical of $R$. Then $Jac(R)$ has a group structure with respect to the operation $a*b=a\cdot b+a+b$ ($a,b\in A$). Therefore $(Jac(R),+,\cdot)$ is a brace.
In general, a ring $R=(R,+,\cdot)$ has a group structure with a multiplication $a*b=a\cdot b+a+b$ is called radical ring. On account of this, a brace is a generalization of radical ring.}
\end{Exa}

\begin{theorem} {\rm (W.Rump[6]) 
Let $(A,+,\cdot)$ be a brace. Then 
\[
  R(a,b):=( \gamma (\gamma(a)(b))^{-1}(a) ,  \gamma(a)(b)  )  \quad (a,b\in A)
\]
 is a nondegenerate unitary YB map.} 
\end{theorem}

Next we introduce a dynamical brace that is the generalization of a brace.
\begin{Def}
{\rm Let $H$ be a non-empty set, $A=(A,+)$ an abelian group with the family of multiplications $\{ \cdot_{\lambda}:A\times A\rightarrow A \}_{\lambda\in H}$ and $\phi$ a map from $H\times A$ to $H$. We call $(A,H,\phi;+,\{ \cdot_{\lambda} \}_{\lambda \in H})$ a dynamical brace(d-brace) if the following conditions are satisfied for all $(\lambda,a,b,c) \in H\times A\times A\times A$. 
\begin{enumerate}
  \item $(a+b)\cdot_{\lambda} c=a\cdot_{\lambda} c +b\cdot_{\lambda} c$ \quad (Right distributive law),
  \item $a\cdot_{\lambda} (b\cdot_{\lambda} c+b+c)= (a\cdot_{\phi(\lambda,c)} b)\cdot_{\lambda} c+a\cdot_{\phi(\lambda,c)} b+a\cdot_{\lambda} c$,
  \item The map $\gamma_{\lambda}(b) :a \mapsto a\cdot_{\lambda} b+a $ is bijective.
\end{enumerate} }
\end{Def}

\begin{Def}
{\rm $(Q,\cdot)$ is a right quasigroup if and only if $Q$ is a non-empty set with a binary operation $(\cdot )$ having the property below for all $a\in Q$
\[
  R(a):Q\rightarrow Q , b\mapsto b\cdot a \quad \mbox{is bijective}.  
\]
A left quasigroups are similarly defined. And a non-empty set $Q$ with left and right quasigroup structure is called simply quasigroup[8]. }
\end{Def}

\begin{Prop}
{\rm Let $H$ be a non-empty set, $A=(A,+)$ an abelian group with a family of right distributive multiplications $\{\cdot_{\lambda}:A\times A\rightarrow A \}_{\lambda\in H}$ and $\phi$ a map from $H\times A$ to $H$. Then $(A,H,\phi;+,\{\cdot_{\lambda} \}_{\lambda\in H})$ is a d-brace if and only if $A$ is a right quasigroup with respect to operations 
\begin{equation}
a*_{\lambda}b:=a\cdot_{\lambda}b+a+b,
\end{equation}
 and satisfies the next condition for all $(\lambda,a,b,c) \in H\times A\times A\times A$
\begin{equation}
  (a*_{\phi(\lambda,c)}b)*_{\lambda}c= a*_{\lambda} (b*_{\lambda}c).
\end{equation} }
\end{Prop}
\begin{proof}
{\rm 1. Let $(A,H,\phi;+,\{\cdot_{\lambda}\}_{\lambda \in H})$ be a d-brace. 
Consider maps $R_{\lambda}(b):a\mapsto a*_{\lambda}b=a\cdot_{\lambda}b+a+b=\gamma_{\lambda}(b)(a)+b$ \ (for all $b\in A$).
Because of bijectivity of $\gamma_{\lambda}(b)$, $R_{\lambda}(b)$ is bijection. Hence $(A,*_{\lambda})$ is right quasigroup. \\
And the relation (14) follows from conditions 1 and 2 of d-brace. \\
2. Suppose that $A$ satisfies the conditions of proposition.
Then the relation (14) implies condition 2 of Definition 3.2,
and bijectivity of $\gamma_{\lambda}(b)$ follows from a right quasigroup structure of $(A,*_{\lambda})$ \qed}
\end{proof}

Note that 
\begin{equation}
 a*_{\lambda}b=a*_{\mu}b \iff a\cdot_{\lambda}b=a\cdot_{\mu}b 
\end{equation}
for all $(\lambda,\mu,a,b)\in H\times H\times A\times A$.

\begin{Prop}
{\rm Let $(A,H,\phi;+,\{\cdot_{\lambda}\}_{\lambda\in H})$ be a d-brace. Then 
\begin{equation}
 \cdot_{\phi(\phi(\lambda,a),b)} = \cdot_{\phi(\lambda, b*_{\lambda}a)}
\end{equation}
as a map from $A\times A$ to $A$, for all $(\lambda,a,b)\in H\times A\times A$.}
\end{Prop}
\begin{proof}
{\rm It follows from the next calculate
\begin{eqnarray*}
(d*_{\phi(\phi(\lambda,a),b)}c)*_{\lambda} (b*_{\lambda}a)
 &=& \{(d*_{\phi(\phi(\lambda,a),b)}c)*_{\phi(\lambda,a)}b\} *_{\lambda}a \\
 &=& \{d*_{\phi(\lambda,a)}(c*_{\phi(\lambda,a)}b)\} *_{\lambda}a \\
 &=& d*_{\lambda} \{c*_{\lambda}(b*_{\lambda}a)\} \\
 &=& (d*_{\phi(\lambda,b*_{\lambda}a)}c)*_{\lambda}(b_{*_{\lambda}}a).
\end{eqnarray*}
Therefore we obtain $d*_{\phi(\phi(\lambda,a),b)}c= d*_{\phi(\lambda,b*_{\lambda}a)}c$. \qed}
\end{proof}

\begin{Prop}
{\rm
Let $(A,H,\phi;+,\{\cdot_{\lambda}\}_{\lambda\in H})$ be a d-brace and $0$ the unit of the abelian group $(A,+)$. Then 
\begin{enumerate}
\item $0\cdot_{\lambda} a=0$ 
\item $a\cdot_{\phi(\lambda,0)} 0=0$  (for all $(\lambda,a) \in H\times A$).
\end{enumerate}}
\end{Prop}
\begin{proof}
{\rm 1. $0\cdot_{\lambda}a=0$ is trivial. \\
2. $a\cdot_{\lambda}0=a\cdot_{\lambda}(0\cdot_{\lambda}0+0+0)=(a\cdot_{\phi(\lambda,0)}0)\cdot_{\lambda}0 +a\cdot_{\phi(\lambda,0)}0 +a\cdot_{\lambda}0$, hence $\gamma_{\lambda}(0)(a\cdot_{\phi(\lambda,0)}0)= (a\cdot_{\phi(\lambda,0)}0)\cdot_{\lambda}0 +a\cdot_{\phi(\lambda,0)}0=0=0\cdot_{\lambda}0+0=\gamma_{\lambda}(0)(0)$. Therefore we obtain $a\cdot_{\phi(\lambda,0)}0=0$ by using bijectivity of $\gamma_{\lambda}(0)$. \qed}
\end{proof}

In generally d-brace has a multiplication $\cdot_{\mu}$ that satisfies 
$a\cdot_{\mu}0 \not=0$ for all $a\in A$
(see Example 5.2).

\begin{Def}
{\rm \begin{enumerate}
\item Let $(A,H,\phi;+,\{\cdot_{\lambda} \}_{\lambda\in H})$ be a d-brace. If a multiplication $\cdot_{\lambda}$ satisfies $a\cdot_{\lambda}0= 0\cdot_{\lambda}a=0$, we call $\cdot_{\lambda}$ zero-symmetry. And if all multiplications of d-brace are zero-symmetry, we call d-brace is zero-symmetry.
\item Let $(A,H,\phi;+,\{\cdot_{\lambda} \}_{\lambda \in H})$ be a d-brace and $K$ a subset of $H$. \\If $(A,K,\phi|_{K\times A};+,\{\cdot_{\lambda} \}_{\lambda\in K})$ is again d-brace, we call it restricted d-brace.

\item Two d-braces $(A,H,\phi;+,\{\cdot_{\lambda} \}_{\lambda \in H} )$ and $(A^{'},H^{'},\phi^{'};+,\{\cdot_{\lambda^{'}} \}_{\lambda^{'} \in H^{'}} )$ are isomorphic if and only if there are bijections $F:A\rightarrow A^{'}$ , $p:H\rightarrow H^{'}$ such that

\begin{enumerate}
\item $F(a+b) = F(a)+F(b)$
\item $F(a\cdot_{\lambda} b) = F(a)\cdot_{p(\lambda)} F(b)$
\item $p \phi = \phi^{'} (p\times F)$ 
\end{enumerate}
(for all $ (\lambda,a,b) \in H\times A\times A$). 
\end{enumerate} }
\end{Def}

Let us reconsider Corollary 2.1.1 stated in Section 2.
Suppose that $A=(A,+)$ is an abelian group, $H$ a non-empty set and $\phi$ a map from $H\times A$ to $H$.
To obtain DYB map associated $A,H,\phi$, we need to construct maps $\mathfrak{L}^{\lambda}_{a}:A\rightarrow A$ that satisfies $\mathfrak{L}_{a}^{\lambda} \cdot \mathfrak{L}_{b}^{\phi(\lambda ,a)}=\mathfrak{L}_{\mathfrak{L}_{a}^{\lambda}(b)+a}^{\lambda}$ (for all $(\lambda,a,b) \in H\times A\times A$). 
The next theorem states a relation between d-braces and DYB maps.
\begin{theorem}
{\rm Let $A=(A,+)$ be an abelian group, $H$ a non-empty set and $\phi$ a map from $H\times A$ to $H$.
\begin{enumerate}
\item Let $(A,H,\phi;+,\{\cdot_{\lambda} \}_{\lambda\in H})$ be a d-brace, then $\{ \mathfrak{L}^{\lambda}_{a}:=\gamma_{\lambda}(a) :A\rightarrow A \}_{(\lambda,a)\in H\times A}$ is a family of isomorphisms of the abelian group $(A,+)$ that satisfies $\mathfrak{L}_{a}^{\lambda} \cdot \mathfrak{L}_{b}^{\phi(\lambda ,a)}=\mathfrak{L}_{\mathfrak{L}_{a}^{\lambda}(b)+a}^{\lambda}$ (for all $(\lambda,a,b) \in H\times A\times A$).
\item Let $\{ \mathfrak{L}^{\lambda}_{a}:A\rightarrow A \}_{(\lambda,a)\in H\times A}$ be a family of isomorphisms of the abelian group $(A,+)$ that satisfies $\mathfrak{L}_{a}^{\lambda} \cdot \mathfrak{L}_{b}^{\phi(\lambda ,a)}=\mathfrak{L}_{\mathfrak{L}_{a}^{\lambda}(b)+a}^{\lambda}$. Define multiplications on $A$ by $a\cdot_{\lambda}b:=\mathfrak{L}^{\lambda}_{b}(a)-a$, then $(A,H,\phi;+,\{\cdot_{\lambda} \}_{\lambda\in H})$ is a d-brace (for all $(\lambda,a,b) \in H\times A\times A$).
\item The correspondence between 1 and 2 is one-to-one.
\end{enumerate}}
\end{theorem}
\begin{proof}
{\rm 1. We prove that $\{ \mathfrak{L}^{\lambda}_{a}:=\gamma_{\lambda}(a) :A\rightarrow A \}_{(\lambda,a)\in H\times A}$ a family of isomorphisms of abelian group and satisfies $\mathfrak{L}_{a}^{\lambda} \cdot \mathfrak{L}_{b}^{\phi(\lambda ,a)}=\mathfrak{L}_{\mathfrak{L}_{a}^{\lambda}(b)+a}^{\lambda}$ (for all $(\lambda,a,b) \in H\times A\times A$).

The bijectivity of $\mathfrak{L}^{\lambda}_{a}$ follows from the definition, and as a result of right distributivity of d-brace  $\mathfrak{L}^{\lambda}_{a}$ is isomorphism. 
The relation $\mathfrak{L}_{a}^{\lambda} \cdot \mathfrak{L}_{b}^{\phi(\lambda ,a)}=\mathfrak{L}_{\mathfrak{L}_{a}^{\lambda}(b)+a}^{\lambda}$ is proved as follows.
\begin{eqnarray*}
\mbox{LHS}
&=& \gamma_{\lambda}(a)\gamma_{\phi(\lambda,a)}(b)(c) \\
&=& (c\cdot_{\phi(\lambda,a)}b+c)\cdot_{\lambda} a+c\cdot_{\phi(\lambda,a)}b+c \\
&=& (c\cdot_{\phi(\lambda,a)}b)\cdot_{\lambda}a+c\cdot_{\lambda}a+c\cdot_{\phi(\lambda,a)}b+c \\
&=& c\cdot_{\lambda}(b\cdot_{\lambda}a+b+a)+c \\
&=& c\cdot_{\lambda}( \gamma_{\lambda}(a)(b)+a)+c \\
&=& \gamma_{\lambda}( \gamma_{\lambda}(a)(b)+a)(c) \\
&=& \mbox{RHS}. 
\end{eqnarray*}

2. We prove that $(A,H,\phi;+,\{\cdot_{\lambda}\}_{\lambda\in H})$ is d-brace.
By definition of $\mathfrak{L}^{\lambda}_{b}$, multiplication $\cdot_{\lambda}$ satisfies the right distributive law. And as a consequence of $\gamma_{\lambda}(b)(a):= a\cdot_{\lambda}b+a=\mathfrak{L}^{\lambda}_{b}(a)$, we obtain bijectivity of $\gamma_{\lambda}(b)$.
Lastly the relation $a\cdot_{\lambda} (b\cdot_{\lambda} c+b+c)= (a\cdot_{\phi(\lambda,c)} b)\cdot_{\lambda} c+a\cdot_{\phi(\lambda,c)} b+a\cdot_{\lambda} c$ is proved as follows
\begin{eqnarray*}
\mbox{LHS}
&=& a\cdot_{\lambda} (\mathfrak{L}^{\lambda}_{c}(b)+c )\\
&=& \mathfrak{L}^{\lambda}_{\mathfrak{L}^{\lambda}_{c}(b)+c  }(a)-a\\
&=& \mathfrak{L}^{\lambda}_{c}\mathfrak{L}^{\phi(\lambda,c)}_{b}(a)-a\\
&=& \mathfrak{L}_{c}^{\lambda}(a\cdot_{\phi(\lambda,c) }b+a )-a\\
&=& \mathfrak{L}^{\lambda}_{c}(a\cdot_{\phi(\lambda,c)}b )+ \mathfrak{L}^{\lambda}_{c}(a)-a\\
&=& \mbox{RHS}.
\end{eqnarray*}

3. Straightforward. \qed }
\end{proof}

\begin{Cor}
{\rm Let $(A,H,\phi;+,\{\cdot_{\lambda} \}_{\lambda \in H} )$ be a d-brace. Then maps $R(\lambda):A\times A\rightarrow A\times A(\lambda\in H)$ defined by
\begin{equation}
  R(\lambda)(a,b)= ( \mathfrak{R}^{\lambda}_{b}(a) ,  \mathfrak{L}_{a}^{\lambda}(b))
   := ( \gamma_{\lambda} (\gamma_{\lambda}(a)(b))^{-1}(a) ,  \gamma_{\lambda}(a)(b)  )  
\end{equation}
are right nondegenerate unitary DYB map associated with $A,H,\phi$.}
\end{Cor}

\begin{Prop}
{\rm Let $\{ \mathfrak{L}^{\lambda}_{a}:A\rightarrow A \}_{(\lambda,a)\in H\times A}$ be a family of isomorphisms of abelian group $A$ and satisfies
\[
\mathfrak{L}_{a}^{\lambda} \cdot \mathfrak{L}_{b}^{\phi(\lambda ,a)}=\mathfrak{L}_{\mathfrak{L}_{a}^{\lambda}(b)+a}^{\lambda}
  \ (\mbox{for all } (\lambda,a,b) \in H\times A\times A ).
\]
Then $\{ \mathfrak{L}^{\lambda}_{a}:A\rightarrow A \}_{(\lambda,a)\in H\times A}$ satisfies 
\[
 \mathfrak{L}^{\phi(\phi(\lambda,a),b)}_{c}= \mathfrak{L}^{\phi(\lambda,\mathfrak{L}^{\lambda}_{a}(b)+a)}_{c}
 \ (\mbox{for all } (\lambda,a,b) \in H\times A\times A ).
\] }
\end{Prop}
\begin{proof}
{\rm \begin{eqnarray*}
 \mathfrak{L}^{\lambda}_{a} (\mathfrak{L}^{\phi(\lambda,a)}_{b} \mathfrak{L}^{\phi(\phi(\lambda,a),b)}_{c} ) 
 &=& \mathfrak{L}^{\lambda}_{a} \mathfrak{L}^{\phi(\lambda,a)}_{\mathfrak{L}^{\phi(\lambda,a)}_{b}(c)+b}\\
 &=& \mathfrak{L}^{\lambda}_{\mathfrak{L}^{\lambda}_{a}(\mathfrak{L}^{\phi(\lambda,a)}_{b}(c)+b) +a}\\
 &=& \mathfrak{L}^{\lambda}_{\mathfrak{L}^{\lambda}_{a}\mathfrak{L}^{\phi(\lambda,a)}_{b}(c) + \mathfrak{L}^{\lambda}_{a}(b) +a}\\ 
 &=& \mathfrak{L}^{\lambda}_{\mathfrak{L}^{\lambda}_{\mathfrak{L}^{\lambda}_{a}(b)+a}(c) +\mathfrak{L}^{\lambda}_{a}(b) +a } \\
 &=& \mathfrak{L}^{\lambda}_{\mathfrak{L}^{\lambda}_{a}(b)+a}\mathfrak{L}^{\phi(\lambda,\mathfrak{L}^{\lambda}_{a}(b)+a)}_{c} \\
 &=& \mathfrak{L}^{\lambda}_{a} (\mathfrak{L}^{\phi(\lambda,a)}_{b} \mathfrak{L}^{\phi(\lambda,\mathfrak{L}^{\lambda}_{a}(b)+a) }_{c} )
\end{eqnarray*} \qed}
\end{proof}

As a consequence of this corollary, if the map $\mathfrak{L}:H \rightarrow Map(A\times A,A)$ is an injection, $\phi$ satisfies the weight-zero condition (for $\mathfrak{R}^{\lambda}_{a}(b):=(\mathfrak{L}^{\lambda}_{\mathfrak{L}^{\lambda}_{a}(b)})^{-1}(a)$).
\begin{Rem}
{\rm In generally, it seems that natural to assume the d-brace \\$(A,H,\phi;+,\{\cdot_{\lambda}\}_{\lambda\in H})$ satisfies the condition $\cdot_{\lambda}=\cdot_{\mu} \iff \lambda = \mu$ $(\lambda,\mu\in H)$. 
For this reason, the injectivity of $\mathfrak{L}$, with respect to parameter $H$, seems natural
(therefore the weight-zero condition also seems natural). }
\end{Rem}

The next theorem gives a relation between brace and d-brace structures over module(i.e., a relation between some YB maps and DYB maps). 

\begin{theorem}
{\rm Let $G$ be a group. suppose that $A=(A,+)$ is a $G$-module, moreover $(A,+,\cdot)$ is a brace. We denote an action of $\lambda$ by $f_{\lambda}$. Let $\phi$ be a map from $G\times A$ to $G$.
Define multiplication $\cdot_{\lambda} \ (\lambda \in G)$ over $A$ 
\begin{equation}
  a\cdot_{\lambda} b:=f_{\lambda}^{-1}(f_{\phi(\lambda,b)}(a) \cdot f_{\lambda}(b)+ f_{\phi(\lambda,b)}(a))-a 
\end{equation}
for all $a,b \in A$. 
Then $(A,G,\phi;+,\{ \cdot_{\lambda}  \}_{\lambda \in G} )$ is d-brace if and only if the map $\phi:G\times A\rightarrow G$ satisfies
\[
  f_{\phi(\lambda,b*_{\lambda} a)} =f_{\phi(\phi(\lambda,a),b)}  \quad (\mbox{for all } (\lambda,a,b) \in G\times A\times A).
\]
(Here multiplication $*_{\lambda} $ are defined by $ a*_{\lambda} b:=a\cdot_{\lambda} b+a+b $). }
\end{theorem}
\begin{proof}
{\rm Using $a\cdot_{\lambda} b=f_{\lambda}^{-1}(f_{\phi(\lambda,b)}(a) \cdot f_{\lambda}(b)+ f_{\phi(\lambda,b)}(a))-a$ we can express operations $*_{\lambda}$ as follows
\[
a*_{\lambda}b=f_{\lambda}^{-1}(f_{\phi(\lambda,b)}(a)*f_{\lambda}(b))
\]
this multiplication satisfy right distributivity, and $(A,*_{\lambda})$ is a right quasigroup for all $\lambda \in H$.
To obtain theorem we need to see $(a*_{\phi(\lambda,c)}b)*_{\lambda}c= a*_{\lambda} (b*_{\lambda}c)$.
\begin{eqnarray*}
\mbox{LHS}
&=& f_{\lambda}^{-1} ( f_{\phi(\lambda,c)}(a*_{\phi(\lambda,c)}b)*f_{\lambda}(c)  )\\
&=& f_{\lambda}^{-1} ( (f_{\phi(\phi(\lambda,c),b)}(a)* f_{\phi(\lambda,c)}(b) ) *f_{\lambda}(c) )
\end{eqnarray*}
and
\begin{eqnarray*}
\mbox{RHS}
&=& f_{\lambda}^{-1} ( f_{\phi(\lambda,b*_{\lambda}c)}(a)* f_{\lambda}(b*_{\lambda}c) )\\
&=& f_{\lambda}^{-1} ( f_{\phi(\lambda,b*_{\lambda}c)}(a)* (f_{\phi(\lambda,c)}(b) *f_{\lambda}(c))  )
\end{eqnarray*}
hence we obtain theorem by compare of LHS and RHS. \qed}
\end{proof}

\begin{Rem}
{\rm If an action of group $G$ is faithful, a map $\phi$ satisfies 
\begin{equation}
  \phi(\lambda,b*_{\lambda} a)=\phi(\phi(\lambda,a),b)  \quad ((\lambda,a,b)\in G\times A\times A)
\end{equation}
this condition corresponds to the weight-zero condition in DYB map.}
\end{Rem}

\begin{Exa}
{\rm Let $(F,+,\times)$ be any field with a trivial brace structure $\cdot$ . Define an action of $a\in F$ by $f_{a}(b):=a^{2}b$
and define $\phi:F\times F\rightarrow F$ by $\phi(a,b):=f_{a}(b)+a=a(ab+1)$. 
From Theorem 3.3 we obtain $a\cdot_{b} c=\{(bc+1)^{2}-1 \}a $.
Then $\phi$ satisfies $\phi(a,b*_{a} c)=\phi(\phi(a,c),b)$ (i.e., weight-zero condition). 
Hence $(F,F,\phi;+,\{\cdot_{a} \}_{a \in F})$ is a d-brace. \\
The DYB map $R(a) \ (a\in F)$ associated with $F,F,\phi$ which corresponds to this d-brace is as follows.
\[
R(a)(b,c)=( \{a(ab+1)^{2}c+1 \}^{-1}b , (ab+1)^{2}c) 
\]
for all $a,b,c\in F$.}
\end{Exa}

\section{Combinatorial aspects of dynamical braces.}
We give a combinatorial approach to the d-brace, from this we obtain a description of d-brace as some family of subsets. 

\begin{theorem}
{\rm Let $(A,+)$ be an abelian group, $H$ a non-empty set.
\begin{enumerate}
 \item Let $(A,H,\phi;+,\{\cdot_{\lambda}\}_{\lambda\in H})$ be a d-brace. We set a family of subsets $\{S_{\lambda} \}_{\lambda\in H}$, $S_{\lambda}:=\{R_{\lambda}(a):A\rightarrow A, b\mapsto b*_{\lambda}a (a \in A)\} \subset A\rtimes Aut(A)$. Then $S_{\lambda}$ satisfies following conditions.
 \begin{enumerate}
     \item $\forall a\in A$ , $\exists! f \in Aut(A) $ s.t., $(a,f) \in S_{\lambda}$,
     \item $\forall (a,f)\in S_{\lambda}$ , $\exists! \mu \in H$ s.t., $(a,f)^{-1}S_{\lambda}= \{ (a,f)^{-1}(b,g)| (b,g)\in S_{\lambda} \} =S_{\mu} $.
 \end{enumerate}
 We denotes $f \in Aut(A)$ of condition (a) by $f_{\lambda}(a)$.  

\item Let $\{ S_{\lambda}\}_{\lambda \in H}$ be a family of subsets of $A\rtimes Aut(A)$ that satisfies conditions (a) and (b).
   Define multiplications $\{\cdot_{\lambda}\}_{\lambda\in H}$ on $A$ by $a\cdot_{\lambda}b:=f_{\lambda}(b)(a)-a$ and a map $\phi$ from $H\times A$ to $H$ by corresponds $(\lambda,a)$ to $\mu$ which determine in condition (b). Then $(A,H,\phi;+,\{\cdot_{\lambda}\}_{\lambda\in H})$ is d-brace.
\item The correspondence between 1 and 2 is one-to-one.
   
\end{enumerate}}
\end{theorem}

\begin{proof}
{\rm 1. Because of $R_{\lambda}(a)(b)=b*_{\lambda}a=b\cdot_{\lambda}a+b+a=\gamma_{\lambda}(a)(b)+a$ and $\gamma_{\lambda}(a)\in Aut(A)$, we can regard $R_{\lambda}(a)$ as an action of $(a,\gamma_{\lambda}(a))$. Therefore $S_{\lambda}\subset A\rtimes Aut(A)$. 
Next we prove that $\{S_{\lambda} \}_{\lambda\in H}$ satisfies conditions (a) and (b).\\
1.1. Condition (a) follows from the definition of $S_{\lambda}$. \\
1.2. For $R_{\phi(\lambda,a)}(b) \in S_{\phi(\lambda,a)}$, we obtain $R_{\lambda}(a)R_{\phi(\lambda,a)}(b)(c) = (c*_{\phi(\lambda.a)}b)*_{\lambda}a = c*_{\lambda}(b*_{\lambda}a) = R_{\lambda}(R_{\lambda}(a)(b))(c)$ for all $b\in A$.
Because of $R_{\lambda}(a)$ is bijection, we obtain following an equality
\begin{eqnarray*}
R_{\lambda}(a) S_{\phi(\lambda,a)}
&=& \{R_{\lambda}(a)R_{\phi(\lambda,a)}(b)|b\in A\}\\
&=& \{R_{\lambda}(R_{\lambda}(a)(b))| b\in A \}\\
&=& S_{\lambda}. 
\end{eqnarray*}
2. We prove that $(A,*_{\lambda})$ is right quasigroup, and satisfies $(a*_{\phi(\lambda,c)}b)*_{\lambda}c= a*_{\lambda}(b*_{\lambda}c)$.\\
2.1. By definition of multiplications, we obtain $a*_{\lambda}b=f_{\lambda}(b)(a)+b$, therefore $*_{\lambda}$ is an action of $(b,f_{\lambda}(b))$. Hence $(A,*_{\lambda})$ is right quasigroup.\\
2.2. We prove that $(A,*_{\lambda})$ satisfies the relation (14).
 Take $(b,f_{\phi(\lambda,c)(b)}) \in S_{\phi(\lambda,c)}$, $(c,f_{\lambda}(c)) \in S_{\lambda}$, by definition of $\phi$
\[
 (c,f_{\lambda}(c))(b,f_{\phi(\lambda,c)}(b))= (c+f_{\lambda}(c)(b) , f_{\lambda}(c)f_{\phi(\lambda,c)}(b)) \in S_{\lambda}.
\]
Therefore $(c+f_{\lambda}(c)(b) , f_{\lambda}(c)f_{\phi(\lambda,c)}(b))= (c+f_{\lambda}(c)(b), f_{\lambda}( c+f_{\lambda}(c)(b) )  )$ by condition (a). From this we obtain (14) as follows 
\begin{eqnarray*}
  (a*_{\phi(\lambda,c)}b)*_{\lambda}c &=& f_{\lambda}(c)f_{\phi(\lambda,c)}(b)(a)+ f_{\lambda}(c)(b)+c\\
                                      &=& f_{\lambda}(f_{\lambda}(c)(b)+c)(a)+f_{\lambda}(c)(b)+c\\
                                      &=& a*_{\lambda}(b*_{\lambda}c). 
\end{eqnarray*}
3. Straightforward. \qed}
\end{proof}

A subgroup $S$ of $A\rtimes Aut(A)$ is said to be regular if, given any $a\in A$, then for each $b \in A$ there exists a unique $x\in S$ such that $x.a=b$, where . denotes an action of $S$. From this, we express regular subgroup $S=\{(a,f(a)) |a\in A \}$.
\begin{Cor}
{\rm Let $A=(A,+)$ be an abelian group.
\begin{enumerate}
 \item Let $(A,+,\cdot)$ be a brace, then $\{R(a):A\rightarrow A,b\mapsto b*a (a\in A) \}$ is a regular subgroup of $A\rtimes Aut(A)$.    
 \item Let $S$ be a regular subgroups $A\rtimes Aut(A)$. Define a multiplication on $A$ by $a\cdot b:=f(b)(a)-a$, then $(A,+,\cdot)$ is  a brace.  
 \item The correspondence between 1 and 2 is one-to-one.
\end{enumerate}}

\end{Cor}
\begin{proof}
{\rm A case of $\#(H)=1$. \qed}
\end{proof}

\begin{Rem}
{\rm F.Catina and R.Frizz has shown a similar things of this corollary in [9,10].
(In [10] they called an algebra with brace structure a radical circle algebra). }
\end{Rem}

The next proposition is a correspondence of Proposition 3.4 and Proposition 3.6.
\begin{Prop}
{\rm Let $\{S_{\lambda} \}_{\lambda\in H}$ be a family of subsets of $A\rtimes Aut(A)$ and satisfies the conditions of Theorem 4.1. Then 
$\{S_{\lambda} \}_{\lambda\in H}$ satisfies    
\begin{description}
\item $S_{\phi(\phi(\lambda,a),b)} = S_{\phi(\lambda,f_{\lambda}(a)(b)+a )}$ (for all $(\lambda,a,b)\in H\times A\times A)$.
\end{description}}
\end{Prop}
\begin{proof}
{\rm Because of the condition (b) of the Theorem 4.1, we obtain \\ $(a,f_{\lambda}(a))^{-1}S_{\lambda}=S_{\phi(\lambda,a)}$.
Therefore
\begin{eqnarray*}
S_{\phi(\phi(\lambda,a),b)} &=& (b,f_{\lambda}(b))^{-1}S_{\phi(\lambda,a)}\\
                            &=& (b,f_{\lambda}(b))^{-1} \{(a,f_{\lambda}(a))^{-1} S_{\lambda}\} \\
                            &=& \{(b,f_{\lambda}(b))^{-1} (a,f_{\lambda}(a))^{-1} \} S_{\lambda}\\
                            &=& (f_{\lambda}(a)(b)+a, f_{\lambda}(a)f_{\lambda}(b))^{-1} S_{\lambda}\\
                            &=& S_{\phi(\lambda,f_{\lambda}(a)(b)+a )}.
\end{eqnarray*} \qed}
\end{proof}

\section{Graphs of dynamical braces and properties.}
Let $(A,+)$ be an abelian group, $H$ a non-empty set and $\{S_{\lambda}\}_{\lambda\in H}$ a family of subsets of $A\rtimes Aut(A)$ which satisfies the conditions (a) and (b) of Theorem 4.1. Here $S_{\lambda}$ is defined by $S_{\lambda}=\{ (a,f_{\lambda}(a)) |a\in A\}$. Then by the condition (b) we obtain a directed edge from $S_{\lambda}$ to $S_{\mu}$, where $\mu$ is $\phi(\lambda,a)$ by definition of $\phi$.

\[\xymatrix{
S_{\lambda} \ar[rr]^{a} & & S_{\mu} \\
}\]

Namely, this graph consists of
\[
    V(A)=\{S_{\lambda}| \lambda\in H \} \quad\mbox{(vertex set)}, 
\]
\[
    E(A)= \{(S_{\lambda},S_{\phi(\lambda,a)} ) | \lambda\in H, a\in A\} \quad\mbox{(edge set)}.
\]
We call this graph associated with $(A,+,\{S_{\lambda}\}_{\lambda \in H})$ a graph of d-brace.

As a consequence of this, 
$S_{\lambda}$ corresponds to multiplication $\cdot_{\lambda}$ (i.e., dynamical parameters correspond to vertices of graph), 
map $\phi$ means a connection of edges, and $\# (A)$ is a degree of graph. This graph has following properties.

\begin{Prop}
{\rm 
\begin{enumerate}
\item Each vertex $S_{\phi(\lambda,a)}$ $(\lambda\in H,a\in A )$ has a loop. Namely $(S_{\phi(\lambda,a)},S_{\phi(\lambda,a)}) \in E(A)$ for all $(\lambda,a) \in H\times A$. 
\item The edge $(S_{\phi(\lambda,a)},S_{\phi(\phi(\lambda,a),b)} ) \in E(A)$  has an inverse edge.
Namely $(S_{\phi(\phi(\lambda,a),b)} ,S_{\phi(\lambda,a)}) \in E(A)$ for all $(\lambda,a,b) \in H\times A\times A$.
\item For the edge $(S_{\phi(\lambda,a)},S_{\phi(\lambda,a)}) \in E(A)$. A corresponding multiplication $\cdot_{\phi(\lambda,a)}$ has zero-symmetry. Hence all d-braces include a zero-symmetry restricted d-brace. 
\item Two isomorphic d-braces give the same underlying graph. 
\end{enumerate}}
\end{Prop}
\begin{proof}
{\rm 1. By Proposition 4.1. 
\[
(S_{\phi(\lambda,a)},S_{\phi(\lambda,a)}) = (S_{\phi(\lambda,a)},S_{\phi(\phi(\lambda,a),0))}) \in E(A). 
\]
2. By definition $(S_{\phi(\phi(\lambda,a),b)} , S_{\phi( \phi(\phi(\lambda,a),b) , f_{\phi(\lambda,a)}(b)^{-1}(-b)}) \in E(A)  $, and 
\[
S_{\phi( \phi(\phi(\lambda,a),b) , f_{\phi(\lambda,a)}(b)^{-1}(-b))} =S_{\phi(\phi(\lambda,a) , 0)}= S_{\phi(\lambda,a)},
\]
 follows from Proposition 4.1. Therefore
 $(S_{\phi(\phi(\lambda,a),b)} , S_{\phi(\lambda,a)}) \in E(A)$.

3.  Follows from definition of $\cdot_{\lambda}$ and Proposition 3.5. For the latter, restrict a set of dynamical parameters to $Im\phi$. Elements of $Im\phi$ correspond to vertices with loop.

4. Let $(A,H,\phi;+_{A},\{ \cdot_{\lambda}\}_{\lambda\in H} )$ and $(B,I,\psi;+_{B},\{\cdot_{\mu}\}_{\mu\in I})$ be two isomorphic d-braces. By definition of isomorphisms. There are maps $F:A\rightarrow A^{'}$, $p:H\rightarrow H^{'}$ such that $F,p$ satisfies the conditions of definition 4.1. 
If $(\lambda,\phi(\lambda,a))\in E(A)$, then $(p(\lambda), p\phi(\lambda,a))=(p(\lambda), \psi(p(\lambda),F(a)) ) \in E(B)$.
Therefore we obtain a bijection between two graphs. \qed}
\end{proof}

Finally we give some examples of graphs and d-braces.
In below examples, the edge $\longleftrightarrow$ means an edge with their inverese edge and $\Longrightarrow$ means double edges (for loops we only use $\rightarrow$). When a graph is complicated, we omit labels of edges. 
 
\begin{Exa}
{\rm Let $A$ be an abelian group then $A$ itself is a regular subgroup of $A\rtimes Aut(A)$.
This regular subgroup corresponds to the trivial brace structure on $A$. (See Example 3.1).}
\end{Exa}

\begin{Exa}
{\rm Set $A=\{0,1,2 \}=\mathbb{Z}_{3}$, $Aut(A)=\{id_{A}, \tau \} , \tau:(0,1,2)\mapsto (0,2,1)$ and 
\[
A\rtimes Aut(A)=\{I=(0,id_{A}) , (0,\tau) , (1,id_{A}) , (1,\tau) , (2,id_{A}) , (2,\tau) \}.
\]
Then next families of subsets of $A\rtimes Aut(A)$ satisfy the conditions (a) and (b).
$(S_{\lambda_{1}},S_{\lambda_{2}},S_{\lambda_{3}},S_{\lambda_{4}})$,
$(S_{\lambda_{1}},S_{\lambda_{2}},S_{\lambda_{3}},S_{\lambda_{5}})$,
$(S_{\lambda_{1}},S_{\lambda_{2}},S_{\lambda_{3}},S_{\lambda_{6}})$.
Here 
\[
  S_{\lambda_{1}}:=\{I , (1,\tau) , (2,\tau) \} ,\quad S_{\lambda_{2}}:=\{I , (1,id_{A}) , (2,\tau) \} , 
\]
\[
  S_{\lambda_{3}}:=\{I , (1,\tau) , (2,id_{A}) \} ,\quad S_{\lambda_{4}}:=\{(0,\tau) , (1,id_{A}) , (2,id_{A}) \},
\]
\[
   S_{\lambda_{5}}:=\{(0,\tau) , (1,\tau) , (2,id_{A}) \} ,\quad S_{\lambda_{6}}:=\{(0,\tau) , (1,id_{A}) , (2,\tau) \}.
\]
In this case sets of dynamical parameters is $H=\{\lambda_{1}, \lambda_{2}, \lambda_{3}, \lambda_{i} \}$ ($i=4,5,6$).

Multiplications that correspond to $S_{\lambda_{1}},S_{\lambda_{2}},S_{\lambda_{3}},S_{\lambda_{4}},S_{\lambda_{5}},S_{\lambda_{6}}$ and graphs of $(S_{\lambda_{1}},S_{\lambda_{2}},S_{\lambda_{3}},S_{\lambda_{i}})$ is as follows (these three graphs give same graph).

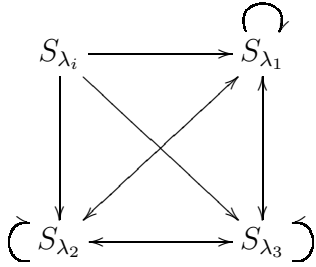
\[\xymatrix{
 S_{\lambda_{i}} \ar[dd] \ar[rr] \ar[ddrr] && S_{\lambda_{1}} \ar@(ul,ur) & \\
 \\
S_{\lambda_{2}} \ar@(dl,ul) \ar@{<->}[rr] \ar@{<->}[rruu] & & S_{\lambda_{3}} \ar@{<->}[uu] \ar@(dr,ur) \\
}\]

\begin{center}
\begin{tabular}{c|cccc c|cccc c|ccc}
$\cdot_{\lambda_{1}}$ & 0 & 1 & 2 & $\quad\quad\quad$ & $\cdot_{\lambda_{2}}$ & 0 & 1 & 2 & $\quad\quad\quad$ &$\cdot_{\lambda_{3}}$ & 0 & 1 & 2  \\ \cline{1-4} \cline{6-9} \cline{11-14}
0 & 0 & 0 & 0 & $\quad\quad\quad$ & 0 & 0 & 0 & 0 & $\quad\quad\quad$ & 0 & 0 & 0 & 0 \\
1 & 0 & 1 & 1 & $\quad\quad\quad$ & 1 & 0 & 0 & 1 & $\quad\quad\quad$ & 1 & 0 & 1 & 0 \\
2 & 0 & 2 & 2 & $\quad\quad\quad$ & 2 & 0 & 0 & 2 & $\quad\quad\quad$ & 2 & 0 & 2 & 0 \\
\end{tabular}
\end{center}

\begin{center}
\begin{tabular}{c|cccc c|cccc c|ccc}
$\cdot_{\lambda_{4}}$ & 0 & 1 & 2 & $\quad\quad\quad$ & $\cdot_{\lambda_{5}}$ & 0 & 1 & 2 & $\quad\quad\quad$ & $\cdot_{\lambda_{6}}$ & 0 & 1 & 2 \\ \cline{1-4} \cline{6-9} \cline{11-14}
0 & 0 & 0 & 0 & $\quad\quad\quad$ & 0 & 0 & 0 & 0 & $\quad\quad\quad$ & 0 & 0 & 0 & 0 \\
1 & 1 & 0 & 0 & $\quad\quad\quad$ & 1 & 1 & 1 & 0 & $\quad\quad\quad$ & 1 & 1 & 0 & 1 \\
2 & 2 & 0 & 0 & $\quad\quad\quad$ & 2 & 2 & 2 & 0 & $\quad\quad\quad$ & 2 & 2 & 0 & 2 \\
\end{tabular}
\end{center}
Therefore d-braces corresponding to $(S_{\lambda_{1}},S_{\lambda_{2}},S_{\lambda_{3}},S_{\lambda_{4}})$,$(S_{\lambda_{1}},S_{\lambda_{2}},S_{\lambda_{3}},S_{\lambda_{5}})$ and $(S_{\lambda_{1}},S_{\lambda_{2}},S_{\lambda_{3}},S_{\lambda_{6}})$ are not isomorphic.
(Therefore the inverse of Proposition 5.1 (4) is not true.)

Moreover in this example, triplet $(S_{\lambda_{1}}, S_{\lambda_{2}}, S_{\lambda_{3}})$ again satisfies conditions (a) and (b).
From this we obtain a subgraph as follows.\\

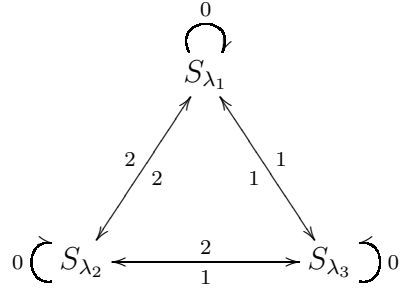
\[\xymatrix{
& S_{\lambda_{1}} \ar@(ul,ur)^{0} & \\
 \\
S_{\lambda_{2}} \ar@(dl,ul)^{0} \ar@{<->}[rr]_{1}^{2} \ar@{<->}[ruu]_{2}^{2} & & S_{\lambda_{3}} \ar@{<->}[luu]_{1}^{1} \ar@(dr,ur)_{0} \\
}\]

It means that the d-brace corresponds to  $(S_{\lambda_{1}}, S_{\lambda_{2}}, S_{\lambda_{3}})$ is a restricted d-brace of d-braces that corresponds to $(S_{\lambda_{1}}, S_{\lambda_{2}}, S_{\lambda_{3}}, S_{\lambda_{i}})$, $i=4,5,6$.}

\end{Exa}

\begin{Exa}
{\rm Set $A=\{(0,0) , (0,1) , (1,0) , (1,1) \}=\mathbb{Z}_{2}\times \mathbb{Z}_{2}$. Let $\tau$ and $\pi$ be automorphisms of $A$ defined by $\tau:((0,1) , (1,0) , (1,1) )\mapsto ((0,1) , (1,1) , (1,0))$, $\pi:((0,1) , (1,0) , (1,1))\mapsto ((1,0) , (0,1) , (1,1))$. Then
\[
S_{\lambda_{1}}=\{ I=((0,0), id_{A}) , ((0,1),\tau) , ((1,0),\tau) , ((1,1), id_{A}) \} 
\]
\[
S_{\lambda_{2}}=\{ I , ((0,1),\tau) , ((1,0),id_{A}) , ((1,1),\tau)  \}
\]
satisfy the conditions (a), (b).
And a set  
\[
S_{\lambda_{3}}=\{ I , ((0,1),\pi) , ((1,0),\pi) , ((1,1),id_{A}) \}.
\]
satisfies the conditions (a),(b). The graphs of $(S_{\lambda_{1}}, S_{\lambda_{2}})$ and $(S_{\lambda_{3}})$ expressed as follows
(because of $S_{\lambda_{3}} \not\simeq A$, $S_{\lambda_{3}}$ corresponds to non-trivial brace).

\[\xymatrix{
S_{\lambda_{1}} \ar@(dl,ul)^{(0,0),(1,1)} \ar@{<=>}[rr]_{(0,1),(1,0)}^{(0,1),(1,1)} & & S_{\lambda_{2}} \ar@(dr,ur)_{(0,0),(1,0)} \\
}\]

\[\xymatrix{
& S_{\lambda_{3}} \ar@(ul,ur)^{(0,0),(0,1),(1,0),(1,1)} & \\
}\]

Multiplications $\cdot_{\lambda_{1}}$, $\cdot_{\lambda_{2}}$, $\cdot_{\lambda_{3}}$ that correspond to $S_{\lambda_{1}}$, $S_{\lambda_{2}}$, $S_{\lambda_{3}}$ is as follows.

\begin{center}
\begin{tabular}{c|ccccc c|cccc}
$\cdot_{\lambda_{1}}$ & (0,0) & (0,1) & (1,0) & (1,1) &$\quad \ $ & $\cdot_{\lambda_{2}}$ & (0,0) & (0,1) & (1,0) & (1,1) \\ \cline{1-5} \cline{7-11}
(0,0) & (0,0) & (0,0) & (0,0) & (0,0) &$\quad \ $ & (0,0) & (0,0) & (0,0) & (0,0) & (0,0) \\
(0,1) & (0,0) & (0,0) & (0,0) & (0,0) &$\quad \ $ & (0,1) & (0,0) & (0,0) & (0,0) & (0,0) \\
(1,0) & (0,0) & (0,1) & (0,1) & (0,0) &$\quad \ $ & (1,0) & (0,0) & (0,1) & (0,0) & (0,1) \\
(1,1) & (0,0) & (0,1) & (0,1) & (0,0) &$\quad \ $ & (1,1) & (0,0) & (0,1) & (0,0) & (0,1) \\
\end{tabular}
\end{center}

\begin{center}
\begin{tabular}{c|cccc}
$\cdot_{\lambda_{3}}$ & (0,0) & (0,1) & (1,0) & (1,1) \\ \cline{1-5}
(0,0) & (0,0) & (0,0) & (0,0) & (0,0) \\
(0,1) & (0,0) & (1,1) & (1,1) & (0,0) \\
(1,0) & (0,0) & (1,1) & (1,1) & (0,0) \\
(1,1) & (0,0) & (0,0) & (0,0) & (0,0) \\
\end{tabular}
\end{center}

Other (more complicated) example.
Let $\tau$ be above automorphism and $\sigma$ an automorphism of $A$ defined by $((0,1),(1,0),(1,1)) \mapsto ((1,0),(1,1),(0,1))$ then
\[
S_{\mu_{1}}=\{I=((0,0),id_{A}), ((0,1),\tau), ((1,0),\sigma), ((1,1),id_{A})\},
\]
\[
S_{\mu_{2}}=\{I, ((0,1),\tau), ((1,0),\tau\sigma), ((1,1),\tau) \} ,
\]
\[
S_{\mu_{3}}=\{I, ((0,1),\sigma^{-1}), ((1,0),\tau\sigma), ((1,1),\sigma^{-1})\},
\]
\[
S_{\mu_{4}}=\{I, ((0,1),\sigma), ((1,0),\tau), ((1,1),id_{A}) \}.
\]
satisfy the conditions (a) and (b). 
The graph of this pair and correspondence multiplication is as follows.

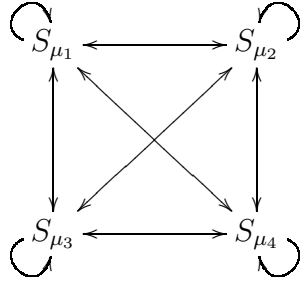
\[\xymatrix{
 S_{\mu_{1}} \ar@(l,u) \ar@{<->}[dd] \ar@{<->}[rr] \ar@{<->}[ddrr] && S_{\mu_{2}} \ar@(r,u) & \\
 \\
S_{\mu_{3}} \ar@(l,d) \ar@{<->}[rr] \ar@{<->}[rruu] & & S_{\mu_{4}} \ar@{<->}[uu] \ar@(r,d) \\
}\]

\begin{center}
\begin{tabular}{c|ccccc c|cccc}
$\cdot_{\mu_{1}}$ & (0,0) & (0,1) & (1,0) & (1,1) &$\quad \ $ & $\cdot_{\mu_{2}}$ & (0,0) & (0,1) & (1,0) & (1,1) \\ \cline{1-5} \cline{7-11}
(0,0) & (0,0) & (0,0) & (0,0) & (0,0) &$\quad \ $ & (0,0) & (0,0) & (0,0) & (0,0) & (0,0) \\
(0,1) & (0,0) & (0,0) & (1,1) & (0,0) &$\quad \ $ & (0,1) & (0,0) & (0,0) & (1,0) & (0,0) \\
(1,0) & (0,0) & (0,1) & (0,1) & (0,0) &$\quad \ $ & (1,0) & (0,0) & (0,1) & (0,0) & (0,1) \\
(1,1) & (0,0) & (0,1) & (1,0) & (0,0) &$\quad \ $ & (1,1) & (0,0) & (0,1) & (1,0) & (0,1) \\
\end{tabular}
\end{center}

\begin{center}
\begin{tabular}{c|ccccc c|cccc}
$\cdot_{\mu_{3}}$ & (0,0) & (0,1) & (1,0) & (1,1) &$\quad \ $ & $\cdot_{\mu_{4}}$ & (0,0) & (0,1) & (1,0) & (1,1) \\ \cline{1-5} \cline{7-11}
(0,0) & (0,0) & (0,0) & (0,0) & (0,0) &$\quad \ $ & (0,0) & (0,0) & (0,0) & (0,0) & (0,0) \\
(0,1) & (0,0) & (1,0) & (1,0) & (1,0) &$\quad \ $ & (0,1) & (0,0) & (1,1) & (0,0) & (0,0) \\
(1,0) & (0,0) & (1,1) & (0,0) & (1,1) &$\quad \ $ & (1,0) & (0,0) & (0,1) & (0,1) & (0,0) \\
(1,1) & (0,0) & (0,1) & (1,0) & (0,1) &$\quad \ $ & (1,1) & (0,0) & (1,0) & (1,0) & (0,0) \\
\end{tabular}
\end{center} }

\end{Exa}

\section*{Acknowledgment}
The author wants to thank Professor Youichi Shibukawa whose insightful comments were invaluable for this study.
He also thanks Professor Kimio Ueno whose enormous support, and members of Uenofs laboratory for useful advice and discussions.


%% References


\bibliographystyle{model3a-num-names}
\bibliography{<your-bib-database>}

%% not want to use model3a-num-names.bst.





% \end{thebibliography}

\end{document}